\documentclass[12pt]{amsart}
\usepackage{amsfonts}
\usepackage{amsmath}
\usepackage{amssymb}
\usepackage{amscd}
\usepackage{thmdefs}
\theoremstyle{definition}
\theoremstyle{remark}
\theoremstyle{theoremtB}
\theoremstyle{theoremt41}
\theoremstyle{theoremtA}
\theoremstyle{theoremt61}
\numberwithin{equation}{section}

\newtheorem*{theoremtB}{Theorem \ref{tB}}
\newtheorem*{theoremt41}{Theorem \ref{t41}}
\newtheorem*{theoremtA}{Theorem \ref{tA}}
\newtheorem*{theoremt61}{Theorem \ref{t61}}

\begin{document}
\title[$C_0$-positive CR torsion solitons]{$C_0$-positivity and a
classification of closed three-dimensional CR torsion solitons}
\author{$^{\dag }$Huai-Dong Cao$^{1}$}
\address{$^{1}$Department of Mathematics, Lehigh University, Bethlehem, PA
18015, USA}
\email{huc2@lehigh.edu }
\author{$^{\ast }$Shu-Cheng Chang$^{2}$}
\address{$^{2}$Department of Mathematics and Taida Institute for
Mathematical Sciences (TIMS), National Taiwan University, Taipei 10617,
Taiwan, R.O.C.}
\email{scchang@math.ntu.edu.tw }
\author{$^{\ast }$Chih-Wei Chen$^{3}$}
\address{$^{3}$Department of Applied Mathematics, National Sun Yat-sen
University, Kaohsiung, Taiwan, R.O.C.}
\email{BabbageTW@gmail.com, chencw@math.nsysu.edu.tw}
\thanks{$^{\dag }$Research partially supported  by a Simons Foundation Collaboration Grant}
\thanks{$^{\ast }$Research supported in part by the MOST of Taiwan}

\begin{abstract}
A closed CR 3-manifold is said to have $C_{0}$-positive pseudohermitian
curvature if $(W+C_{0}Tor)(X,X)>0$ for any $0\neq X\in T_{1,0}(M)$. We
discover an obstruction for a closed CR 3-manifold to possess $C_{0}$%
-positive pseudohermitian curvature. We classify closed three-dimensional CR
Yamabe solitons according to $C_{0}$-positivity and $C_{0}$-negativity
whenever $C_{0}=1$ and the potential function lies in the kernel of Paneitz
operator. Moreover, we show that any closed three-dimensional CR torsion
soliton must be the standard Sasakian space form. At last, we discuss the
persistence of $C_{0}$-positivity along the CR torsion flow starting from a
pseudo-Einstein contact form.
\end{abstract}

\maketitle

\section{Introduction}

Self-similar solutions, also known as geometric solitons, of various
geometric flows have attracted lots of attentions in recent years because
their close ties with singularity formations in the flows. In particular,
important progress has been made in the study of Ricci solitons,
self-similar solutions to the mean curvature flow, as well as Yamabe
solitons, etc.

There have been several flows proposed to investigate the geometry and
topology of CR manifolds, such as CR Yamabe flow, $Q$-curvature flow, CR
Calabi flow, and CR torsion flow. Among them, CR torsion flow behaves more like the Riemannian Ricci flow. Similar to their Riemannian
counterparts, for the CR Yamabe flow it is relatively easier to prove long time existence and convergence results than the CR torsion flow, but it provides less information about the local geometry.

The CR torsion flow is to deform the CR contact form $\theta $ and the
complex structure $J$ by Tanaka-Webster curvature $W$ and torsion $A$
respectively. Namely, 
\begin{equation}
\left\{ 
\begin{array}{l}
\frac{\partial J}{\partial t}=2A_{J,\theta }, \\ 
\frac{\partial \theta }{\partial t}=-2W\theta.%
\end{array}%
\right.  \label{0}
\end{equation}%
Note that the second equation alone is called the CR Yamabe flow, where to a
certain extent  $J$ is freely changed. So the first equation is added to confine the
behavior of $J$ and the CR torsion tensor $A$. In the paper \cite{cw}, it was shown that there exists a unique smooth solution to the
CR torsion flow (\ref{0}) in a small time interval with a certain CR
pluriharmonic function as the initial data. In spirit, it is the CR
analogue of the Cauchy-Kovalevskaya local existence and uniqueness theorem
for analytic partial differential equations associated with Cauchy initial
value problems. Unlike its Riemannian analogue, the problem of asymptotic
convergence of three-dimensional CR torsion flows is widely open (\cite{ckw}%
). Inspired by Perelman's work \cite{pe1} and \cite{pe2} on Hamilton's Ricci
flow and its solitons, we introduce the analogous notion of Ricci solitons -
CR torsion solitons for the CR torsion flow. The structure of CR torsion
solitons may be a necessary ingredient towards understanding the asymptotic
convergence of solutions of the CR torsion flow (\ref{0}). Indeed, one
expects CR torsion solitons to model singularity formations of the CR
torsion flow. The equations of three-dimensional CR torsion solitons are 
%\end{subequations}
\begin{equation}
\left\{ 
\begin{array}{l}
W+\frac{1}{2}f_{0}=\mu , \\ 
f_{11}+iA_{11}f=-A_{11},%
\end{array}%
\right.  \label{1B}
\end{equation}%
and the ones for CR Yamabe solitons are 
\begin{equation}
\left\{ 
\begin{array}{l}
W+\frac{1}{2}f_{0}=\mu , \\ 
f_{11}+iA_{11}f=0,%
\end{array}%
\right.  \label{2B}
\end{equation}%
for some smooth function $f$ on a pseudohermitian $3$-manifold $%
(M^{3},J,\theta )$. Note that they are the same when $M$ has vanishing
torsion. In \cite{ccc}, we have derived some classification theorems for CR
Yamabe solitons. Here we study CR torsion solitons and their
classifications, especially under a curvature assumption called $C_{0}$%
-positive.

%We first study the geometry and topology of CR $3$-manifolds with positive pseudohermitian curvature 
%\begin{equation*}
%(W+\frac{1}{2}Tor)(X,X)>0,\ \ \mathrm{for\ any}\ \ X\in T_{1,0}(M)
%\end{equation*}%
%which is need for classification of CR solitons. 
We say that a CR $3$-manifold is \textit{$C_{0}$-positive} if the
pseudohermitian curvature is $C_{0}$-positive, i.e., 
\begin{equation*}
(W+C_{0}Tor)(X,X)>0,\ \ \mathrm{for\ any}\ \ 0\neq X=x^{1}Z_{1}\in
T_{1,0}(M),
\end{equation*}%
where $W(X,X):=Wx^{1}x^{\bar{1}}$ is the Tanaka-Webster operator, sometimes
called the pseudohermitian Ricci curvature, and $Tor(X,X):=2\mathrm{Re}(iA_{%
\bar{1}\bar{1}}x^{\bar{1}}x^{\bar{1}})$ is the pseudohermitian torsion
tensor. This condition is a key to drive the most general CR Li-Yau gradient
estimate that serves as a generalization of the CR analogue of Cao-Yau
gradient estimate (\cite{cy}) in a closed pseudohermitian manifold with such
a nonnegative pseudohermitian curvature condition and non-vanishing torsion
as in \ \cite{ckl}, \cite{cftw} and \cite{ccl}. We obtain the following
result for $C_{0}$-positive CR $3$-manifolds.

\begin{theorem}\label{tB} 
Let $M$ be a closed CR $3$-manifold. For any $C_0\geq 0$, $C_0$%
-positive is equivalent to the curvature-torsion pinching condition $%
W(x)>2C_{0}|A_{11}|(x)$ for all $x\in M $. Moreover, if $M$ is $C_0$%
-positive with $C_{0}\geq \frac{1}{2}$, then $M$ admits a Riemannian metric
of positive scalar curvature.
\end{theorem}

For closed three-dimensional CR Yamabe solitons $(M^{3},J,\theta )$, we
derive the following classification according to $C_0$-(non)positivity.

\begin{theorem}\label{t41} 
\textrm{(i)} Any simply connected closed three-dimensional CR
Yamabe soliton with $(W+Tor)(X,X)>0$ for $0\neq X\in T_{1,0}(M)$ and $%
P_{0}f=0$ must be the standard CR three sphere.

\textrm{(ii)} Any closed three-dimensional CR Yamabe soliton with $%
(W+Tor)(X,X)=0$ for $0\neq X\in T_{1,0}(M)$ and $P_{0}f=0$ must be the
standard Heisenberg space form which is a Seifert fiber space over a
euclidean $2$-orbifold with nonzero Euler number.

\textrm{(iii)} Any closed three-dimensional CR Yamabe soliton with $%
(W+Tor)(X,X)<0$ for $0\neq X\in T_{1,0}(M)$ and nonnegative CR Paneitz
operator must be the standard Lorentz space form which is a Seifert fiber
space over a hyperbolic orbifold with nonzero Euler number.
\end{theorem}

\begin{remark}
The condition $P_0f=0$ is not very restrictive in the sense that the kernel
of the CR Paneitz operator $P_{0}$ is infinite dimensional, containing all
CR-pluriharmonic functions (see \cite{hi}).
\end{remark}

On the other hand, we can classify closed three-dimensional CR torsion
solitons without using the pseudohermitian curvature condition. This is the
CR analogue of the fact that closed three-dimensional Ricci solitons must be
Einstein.

\begin{theorem}\label{tA} 
Every closed three-dimensional CR torsion soliton (\ref{1B}) has
constant Tanaka-Webster curvature and zero torsion $A_{11}=0$. Therefore,
they must belong to one of the following three cases:

\textrm{(i)} The standard CR spherical space form in case $W=1$.

\textrm{(ii)} The standard Heisenberg space form in case $W=0$.

\textrm{(iii)} The standard Lorentz space form in case $W=-1$.
\end{theorem}

In the last section, we derive some basic evolution equations for curvature
and torsion along the CR torsion flow and then obtain the preserving
property for the uniform pinching pseudohermitian curvature condition under
certain assumptions.

\begin{theorem}\label{t61} 
Let $(M^{3},J,\overset{0}{\theta })$ be a closed CR manifold with pseudo-Einstein $\overset{0}{\theta}$. 
If there exists a pluriharmonic function $\gamma(0)$ such that $W(0)$ of 
$\theta (0):=e^{2\gamma (0)}\overset{0}{\theta }$ is pluriharmonic, then the solution $\theta(t)$ of the
CR torsion flow 
\begin{equation*}
\left\{ 
\begin{array}{l}
\frac{\partial J}{\partial t}=2A_{J,\theta }, \\ 
\frac{\partial \theta }{\partial t}=-2W\theta , \\ 
\overset{0}{P}\gamma (t)=0,\ \overset{0}{P}W(t)=0%
\end{array}%
\right.  % \label{torsionflow}
\end{equation*}%
has the following evolution equations 
\begin{equation*}
\left\{ 
\begin{array}{ccl}
\frac{\partial W}{\partial t} & = & 5\Delta _{b}W+2(W^{2}-|A_{11}|^{2}), \\ 
\frac{\partial |A_{11}|^{2}}{\partial t} & = & \Delta
_{b}|A_{11}|^{2}-2|\nabla |A_{11}||^{2}.%
\end{array}%
\right. 
\end{equation*}%
Moreover, if $W>2C_{0}\max |A_{11}|$ initially with $C_{0}\geq \frac{1}{2}$,
then it holds for all $t<T$. In particular, the solution $\theta (t)$ is $%
C_{0}$-positive for all $t<T$, i.e., 
\begin{equation*}
(W+C_{0}Tor)(X,X)>0\ \mbox{ for all non-vanishing}\ X\in T_{1,0}(M)\mbox{
and for all }t<T.
\end{equation*}
\end{theorem}

This theorem is motivated by the flow approach for proving Frankel's conjecture in K\"ahler geometry, e.g. \cite{ct1,ct2}. As a CR analogue of it, we conjecture that the CR torsion flow which starts from a closed CR manifold with pseudo-Einstein $\theta(0)$, positive constant curvature $W(0)$ and $W+\frac{1}{2}Tor(X,X)>0$ must converge to the Sasakian space form.  We believe Theorem \ref{t61} will be useful for proving the conjecture.\\

%The rest of the paper is organized as follows. In Section $2$, we review some basic material of pseudohermitian manifolds. In Section $3$, we study the geometry and topology of CR $3$-manifolds with positive pseudohermitian curvature as in Theorem \ref{tB}. In Section 4, we first explain that CR torsion solitons correspond to self-similar solutions of the CR torsion flow and then obtain the CR Harnack quantity. In Section 5, we derive classification results of CR Yamabe solitons and torsion solitons. Finally we 

\noindent \textbf{Acknowledgments.} Part of the project was done during the
visit of the second author to Yau Mathematical Sciences Center, Tsinghua
University. He would like to express his thanks for the warm hospitality. 
%%%%%%%%%%%%%%%%%%%%%%%%%%%%%%%%%%%

\section{Preliminaries}

%%%%%%%%%%%%%%%%%%%%%%%%%%%%%%%%%%%
We give a brief introduction to pseudohermitian geometry on a closed $3$%
-manifold (see \cite{l1,l2} for more details). Let $M$ be a closed $3 $%
-manifold with an oriented contact structure $\xi $. There always exists a
global contact form $\theta $, obtained by patching together local ones with
a partition of unity. The characteristic vector field of $\theta $ is the
unique vector field $T$ such that ${\theta }(T)=1$ and $\mathcal{L}_{T}{%
\theta }=0$ or $d{\theta }(T,{\cdot })=0$. A CR structure compatible with $%
\xi $ is a smooth endomorphism $J:{\xi }{\rightarrow }{\xi }$ such that $%
J^{2}=-Id$. A pseudohermitian structure compatible with $\xi $ is a CR
-structure $J$ compatible with $\xi $ together with a global contact form $%
\theta $. The CR structure $J$ can extend to $\mathbb{C}\otimes \xi $ and
decomposes $\mathbb{C}\otimes \xi $ into the direct sum of $T_{1,0}$ and $%
T_{0,1}$ which are eigenspaces of $J$ with respect to $i$ and $-i$,
respectively.

Let $\left \{ T,Z_{1},Z_{\bar{1}}\right \} $ be a frame of $TM\otimes 
\mathbb{C}$, where $Z_{1}$ is any local frame of $T_{1,0},\ Z_{\bar{1}}=%
\overline{Z_{1}}\in T_{0,1}$ and $T$ is the characteristic vector field.
Then $\left
\{ \theta ,\theta ^{1},\theta ^{\bar{1}}\right \} $, the coframe
dual to $\left \{ T,Z_{1},Z_{\bar{1}}\right \} $, satisfies 
\begin{equation}
d\theta =ih_{1\bar{1}}\theta ^{1}\wedge \theta ^{\bar{1}},  \label{22}
\end{equation}%
for some positive function $h_{1\bar{1}}$. Actually we can always choose $%
Z_{1}$ such that $h_{1\bar{1}}=1$; hence, throughout this paper, we assume $%
h_{1\bar{1}}=1$.

The Levi form $\left \langle \ ,\ \right \rangle _{L_{\theta }}$ is the
Hermitian form on $T_{1,0}$ defined by%
\begin{equation*}
\left \langle Z,Y\right \rangle _{L_{\theta }}=-i\left \langle d\theta
,Z\wedge \overline{Y}\right \rangle .
\end{equation*}%
We can extend $\left \langle \ ,\ \right \rangle _{L_{\theta }}$ to $T_{0,1}$
by defining $\left \langle \overline{Z},\overline{Y}\right \rangle
_{L_{\theta }}=\overline{\left \langle Z,Y\right \rangle }_{L_{\theta }}$
for all $Z,Y\in T_{1,0}$. The Levi form induces naturally a Hermitian form
on the dual bundle of $T_{1,0}$, denoted by $\left \langle \ ,\
\right
\rangle _{L_{\theta }^{\ast }}$, and hence on all the induced tensor
bundles. Integrating the Hermitian form (when acting on sections) over $M$
with respect to the volume form $d\mu =\theta \wedge d\theta $, we get an
inner product on the space of sections of each tensor bundle. We denote the
inner product by the notation $\left \langle \ ,\ \right \rangle $. For
example, 
\begin{equation}
\left \langle \varphi ,\psi \right \rangle =\int_{M}\varphi \overline{\psi}\
d\mu ,  \label{21}
\end{equation}%
for functions $\varphi $ and $\psi $.

The pseudohermitian connection of $(J,\theta )$ is the connection $\nabla $
on $TM\otimes \mathbb{C}$ (and extended to tensors) given in terms of a
local frame $Z_{1}\in T_{1,0}$ by

\begin{equation*}
\nabla Z_{1}=\theta _{1}{}^{1}\otimes Z_{1},\quad \nabla Z_{\bar{1}}=\theta
_{\bar{1}}{}^{\bar{1}}\otimes Z_{\bar{1}},\quad \nabla T=0,
\end{equation*}%
where $\theta _{1}{}^{1}$ is the $1$-form uniquely determined by the
following equations:

\begin{equation}
\begin{split}
d\theta ^{1}& =\theta ^{1}\wedge \theta _{1}{}^{1}+\theta \wedge \tau ^{1},
\\
\tau ^{1}& \equiv 0\mod \theta^{\bar{1}}, \\
0& =\theta _{1}{}^{1}+\theta _{\bar{1}}{}^{\bar{1}},
\end{split}
\label{id10}
\end{equation}%
where $\tau ^{1}=A^{1}{}_{%
\bar{1}}\theta ^{\bar{1}}$ is the pseudohermitian torsion. Moreover, the structure equation for $\nabla$ is

\begin{equation}
d\theta_{1}{}^{1}=W\theta ^{1}\wedge \theta ^{\bar{1}}+2i \rm{Im}(A^{%
\bar{1}}{}_{1,\bar{1}}\theta ^{1}\wedge \theta ),  \label{id11}
\end{equation}%
where $W$ is the Tanaka-Webster scalar curvature and the index preceded by
a comma denotes the covariant derivative. The
indices $0, 1$ and $\bar{1}$ indicate derivatives with respect to $T, Z_1$
and $Z_{\bar{1}}$. For derivatives of a scalar function, we will often omit
the comma, for instance, $\varphi_{1}=Z_1\varphi,\ \varphi_{1\bar{1}}=Z_{%
\bar{1}}Z_1\varphi-\theta_1^1(Z_{\bar{1}})Z_1\varphi,\ \varphi_{0}=T\varphi$
for a (smooth) function $\varphi$.

For a real function $\varphi $, the subgradient $\nabla _{b}$ is defined by $%
\nabla _{b}\varphi \in \xi $ and $\left \langle Z,\nabla _{b}\varphi
\right
\rangle _{L_{\theta }}=d\varphi (Z)$ for all vector fields $Z$
tangent to the contact plane. Locally $\nabla _{b}\varphi =\varphi _{\bar{1}%
}Z_{1}+\varphi _{1}Z_{\bar{1}}$. We can use the connection to define the
subhessian as the complex linear map

\begin{equation*}
(\nabla ^{H})^{2}\varphi :T_{1,0}\oplus T_{0,1}\rightarrow T_{1,0}\oplus
T_{0,1}
\end{equation*}%
by 
\begin{equation*}
(\nabla ^{H})^{2}\varphi (Z)=\nabla _{Z}\nabla _{b}\varphi .\ \ 
\end{equation*}%
Also 
\begin{equation*}
\Delta _{b}\varphi =Tr\left( (\nabla ^{H})^{2}\varphi \right) =(\varphi _{1%
\bar{1}}+\varphi _{\bar{1}1}).
\end{equation*}%
For all $Z=x^{1}Z_{1}\in T_{1,0}$, we define

\begin{equation*}
\begin{split}
Ric(Z,Z)& =Wx^{1}x^{\bar{1}}=W|Z|_{L_{\theta }}^{2}, \\
Tor(Z,Z)& =2\mathrm{Re}\ iA_{\bar{1}\bar{1}}x^{\bar{1}}x^{\bar{1}}.
\end{split}%
\end{equation*}

Next we recall the definition of CR Paneitz operator.

\begin{definition}
\label{4} Let $(M,J,\theta )$ be a closed three-dimensional pseudohermitian
manifold. We define (\cite{l1}) 
\begin{equation*}
P\varphi =(\varphi _{\bar{1}}{}^{\bar{1}}{}_{1}+iA_{11}\varphi ^{1})\theta
^{1}=(P_{1}\varphi )\theta ^{1},
\end{equation*}%
which is an operator that characterizes CR-pluriharmonic functions. Here $%
P_{1}\varphi =\varphi _{\bar{1}}{}^{\bar{1}}{}_{1}+iA_{11}\varphi ^{1}$ and $%
\overline{P}\varphi =\overline{P_{1}\varphi}\theta ^{\bar{1}}$, the
conjugate of $P$. The CR Paneitz operator $P_{0}$ is defined by%
\begin{equation}
P_{0}\varphi =\left( \delta _{b}(P\varphi )+\overline{\delta }_{b}(\overline{%
P}\varphi )\right) ,  \label{id9}
\end{equation}%
where $\delta _{b}$ is the divergence operator that takes $(1,0)$-forms to
functions by $\delta _{b}(\sigma _{1}\theta ^{1})=\sigma _{1,}{}^{1}$, and
similarly, $\overline{\delta}_{b}(\sigma _{\bar{1}}\theta ^{\bar{1}})=\sigma
_{\bar{1},}{}^{\bar{1}}$.
\end{definition}

We observe that%
\begin{equation}
\int_{M}\langle P\varphi +\overline{P}\varphi ,d_{b}\varphi \rangle
_{L_{\theta }^{\ast }}\ d\mu =-\int_{M}P_{0}\varphi \cdot \varphi \ d\mu 
\label{10}
\end{equation}%
with $d\mu =\theta \wedge d\theta .$ One can check that $P_{0}$ is
self-adjoint, that is, $\left\langle P_{0}\varphi ,\psi \right\rangle
=\left\langle \varphi ,P_{0}\psi \right\rangle $ for all smooth functions $%
\varphi $ and $\psi $. For the details about these operators, the reader can
make reference to \cite{gl}, \cite{hi}, \cite{l1}, \cite{gl} and \cite{fh}.

%\begin{definition}
%\label{P_0>0} On a complete pseudohermitian $3$-manifold $(M,J,\theta ),$ we call the Paneitz operator $P_{0}$ with respect to $(J,\theta )$ essentially positive if there exists a constant $\Lambda $ $>$ $0$ such that 
%\begin{equation*}
%\int_{M}P_{0}\varphi \cdot \varphi d\mu \geq \Lambda \int_{M}\varphi^{2}d\mu .
%\end{equation*}%
%for all real $C^{\infty }$ smooth functions $\varphi $ $\in (\ker P_{0})^{\perp }$ (i.e. perpendicular to the kernel of $P_{0}$ in the $L^{2}$ norm with respect to the volume form $d\mu $ $=$ $\theta \wedge d\theta ).$ We say that $P_{0}$ is nonnegative if 
%\begin{equation*}
%\int_{M}P_{0}\varphi \cdot \varphi d\mu \geq 0
%\end{equation*}%
%for all real $C^{\infty }$ smooth functions.
%\end{definition}

%\begin{remark}
%2. ( \cite{ccch}) Let $(M,J,\theta )$ be a closed three-dimensional pseudohermitian manifold. The essentially positivity of $P_{0}$ is a CR invariant in the sense that it is \ independent of the choice of the contact form $\theta $. Moreover, the corresponding CR Paneitz operator is essentially positive if the torsion is vanishing.
%\end{remark}

%%%%%%%%%%%%%%%%%%%%%%%%%%%%%%%%%%%%%%%%%%%%%%%%

\section{Geometry and Topology of $C_0$-positive CR $3$-manifolds}

%%%%%%%%%%%%%%%%%%%%%%%%%%%%%%%%%%%%%%%%%%%%%%%%

In this section, we study the geometry and topology of CR $3$-manifolds
which are $C_{0}$-positive, namely, for all $0\neq X=x^{1}Z_{1}\in T_{1,0}(M)
$, %\begin{equation*}
%(W+C_0Tor)(X,X)>0,
%\end{equation*}
\begin{align}
(W+C_{0}Tor)(X,X)=& \ Wx^{1}x^{\bar{1}}+2C_{0}\mathrm{Re}[i(A_{\bar{1}\bar{1}%
}x^{\bar{1}}x^{\bar{1}})]  \label{WTor} \\
=& \ Wx^{1}x^{\bar{1}}+C_{0}i(A_{\bar{1}\bar{1}}x^{\bar{1}}x^{\bar{1}%
}-A_{11}x^{1}x^{1})>0.  \notag
\end{align}

Our aim in this section is to prove the following

\begin{theoremtB}
Let $M$ be a closed CR $3$-manifold. For any $C_0\geq 0$, $C_0$-positive is
equivalent to the curvature-torsion pinching condition $%
W(x)>2C_{0}|A_{11}|(x)$ for all $x\in M$. Moreover, if $M$ is $C_0$-positive
with $C_{0}\geq \frac{1}{2}$, then $M$ admits a Riemannian metric of
positive scalar curvature.
\end{theoremtB}

We divide the proof into two lemmas. The first lemma contains results more
than the theorem states.

\begin{lemma}
\label{c41} Let $M$ be a closed CR $3$-manifold. For any $C_0\geq 0$, $C_0$%
-positive curvature condition is equivalent to the pinching condition $%
W(x)>2C_{0}|A_{11}|(x)$ for all $x\in M$. Similarly, $C_0$-negative
curvature condition is equivalent to $W(x)<-2C_{0}|A_{11}|(x)$ for all $x\in
M$.
\end{lemma}

\begin{proof}
Fix a point $x\in M$ and denote $A_{11}(x)=a(x)+b(x)i$. Without loss of
generality, one may consider $X=x^{1}Z_1=(1+si)Z_1$ for some $s\in \mathbb{R}
$. The positivity condition (\ref{WTor}) reads as 
\begin{align*}
W(1+si)(1-si)>&\ -C_0i[(a-bi)(1-si)^2-(a+bi)(1+si)^2] \\
=&\ -C_0[ 4as+2b-2bs^2 ],
\end{align*}
i.e., 
\begin{equation}
W> -4C_0\frac{as+b}{1+s^{2}}+2C_0b.  \label{41b}
\end{equation}
Since $X$ is arbitrary, this inequality holds for all $s\in \mathbb{R}$.
Denote $f(s)=-4\frac{as+b}{1+s^{2}}+2b$. When $a\neq 0$, we have 
\begin{equation}
f(s)\leq \max_{s\in\mathbb{R}} f(s)=f(s_0)=2\sqrt{a^{2}+b^{2}},  \label{41a}
\end{equation}
where $s_{0}=\frac{b+\sqrt{a^{2}+b^{2}}}{-a}$ is a critical number of $f$.
On the other hand, it is easy to see that $f(s)\leq 2|b|$ when $a=0$. 
%\begin{equation*}
%\left\{ 
%\begin{array}{l}
%f(s)\leq |b|,\ \mathrm{if\ }\ a=0, \\ 
%f(s)\leq |a|,\ \ \mathrm{if\ }\ b=0.%
%\end{array}%
%\right.
%\end{equation*}
All these imply $f(s)\leq 2 |A_{11}|$ for all $s\in \mathbb{R}$ and thus $%
W(x)>2C_0|A_{11}|(x)$ for all $x\in M$. Therefore, $C_0$-positive is
equivalent to $W(x)>2C_{0}|A_{11}|(x)$.

Similarly, at the minimum point $s_{1}=\frac{b-\sqrt{a^{2}+b^{2}}}{-a}(a\neq
0),$ we have 
\begin{equation}
f(s)\geq \min_{s\in\mathbb{R}}f(s)=f(s_1)=-2\sqrt{a^{2}+b^{2}}.  \label{42a}
\end{equation}%
Then it follows analogously that $C_0$-negative curvature condition is
equivalent to $W(x)<-2C_{0}|A_{11}|(x)$.
\end{proof}

\begin{remark}
Note that $C_0$-positivity can also be expressed as 
\begin{equation*}
(W-C_{0}Tor)(Y,Y)>0
\end{equation*}%
for all $0\neq Y=y^{1}Z_{1}\in T_{1,0}(M)$ with $y^{1}=ix^{1}$. This is to
say that $C_0$-positive actually implies $W>C_0 |Tor|$, although this does
not imply $W>2C_0|A_{11}|$ literally without using the lemma above.
\end{remark}

Next, it is interesting to know the obstruction for the $C_0$-positivity
when $C_0\geq \frac{1}{2}$. The following lemma completes the proof of
Theorem \ref{tB}.

\begin{lemma}
\label{l42} Let $(M^{3},J,\theta )$ be a closed CR $3$-manifold with $%
W(x)>\left\vert A_{11}(x)\right\vert$. Then $M$ admits a Riemannian metric
of positive scalar curvature.
\end{lemma}

\begin{proof}
We recall that the Webster (adapted) Riemannian metric on $M$ is defined by 
\begin{equation*}
g_{\lambda }=d\theta +\lambda ^{-2}\theta ^{2}
\end{equation*}
for any parameter $\lambda >0.$ \ Now it follows from the paper by Chang and
Chiu (\cite{cchi}) that the Ricci curvature $R_{ij}^{\lambda }$ with respect
to $g_{\lambda }$ is 
\begin{equation*}
\begin{split}
R_{11}^{\lambda }& =2W-2\lambda ^{-2}-2i\lambda ^{2}\mathrm{Im}A_{\bar{1}%
\bar{1}}\theta _{1}{}^{1}(T)+2\mathrm{Im}A_{\bar{1}\bar{1}}-\lambda ^{2}%
\mathrm{Re}A_{\bar{1}\bar{1},0} \\
R_{22}^{\lambda }& =2W-2\lambda ^{-2}+2i\lambda ^{2}\mathrm{Im}A_{\bar{1}%
\bar{1}}\theta _{1}{}^{1}(T)-2\mathrm{Im}A_{\bar{1}\bar{1}}+\lambda ^{2}%
\mathrm{Re}A_{\bar{1}\bar{1}},_{0} \\
R_{33}^{\lambda }& =-2\lambda ^{2}\left\vert A_{\bar{1}\bar{1}%
}\right\vert^{2}+2\lambda ^{-2} \\
R_{12}^{\lambda }& =2i\lambda ^{2}\mathrm{Re}A_{\bar{1}\bar{1}%
}\theta_{1}{}^{1}(T)-2\mathrm{Re}A_{\bar{1}\bar{1}}-\lambda ^{2}\mathrm{Im}%
A_{\bar{1}\bar{1},0} \\
R_{13}^{\lambda }& =2\lambda \mathrm{Re}A_{11,\bar{1}},\text{ }%
R_{23}^{\lambda}=-2\lambda \mathrm{Im}A_{11,\bar{1}}.
\end{split}%
\end{equation*}%
and then the scalar curvature is 
\begin{equation*}
R^{\lambda }=4W-2\lambda ^{2}\left\vert A_{11}\right\vert ^{2}-2\lambda^{-2}.
\end{equation*}

We want to find $\lambda$ so that $R$ is positive. This is not always
solvable. For instance, when $W=|A_{11}|=0$, $R^{\lambda}$ is always
negative. So we need to find an admissible condition on $W$ and $|A_{11}|$.
Let $\mu=\lambda^2$. Then $R^{\lambda}>0$ if and only if the quadratic
polynomial $2|A_{11}|^2\mu^2-4W\mu+2$ attains negative values, i.e., the
coefficients should satisfy 
\begin{equation*}
0> (4W)^2-4(2|A_{11}|^2 )(2)= 16(W^2-|A_{11}|^2).
\end{equation*}
It follows that if $W>\left\vert A_{11}\right\vert$, then there is a
positive constant $\lambda$ such that $R^{\lambda }>0$. Although $\lambda$
varies pointwisely, there exists a uniform $\lambda>0$ on the closed
manifold $M$. This completes the proof.
\end{proof}

%%%%%%%%%%%%%%%%%%%%%%%%%%%%%%%%%%%%%%%%%%%%%%%%

\section{CR Harnack Quantity}

%%%%%%%%%%%%%%%%%%%%%%%%%%%%%%%%%%%%%%%%%%%%%%%%

For CR manifolds, we have the concept of CR Yamabe solitons (\ref{2A}) which
are self-similar solutions to the CR Yamabe flow on a pseudohermitian $(2n+1)$%
-manifold. %\begin{equation}
%\left\{ 
%\begin{array}{l}
%\frac{\partial }{\partial t}\theta (t)=-2W(t)\theta \left( t\right) , \\ 
%\theta \left( 0\right) =\mathring{\theta}.%
%\end{array}%
%\right.  \label{00}
%\end{equation}%
Similarly, one can introduce CR torsion solitons which correspond to
self-similar solutions to the CR torsion flow (\ref{0}) on a pseudohermitian 
$(2n+1)$-manifold.

\begin{definition}
A pseudohermitian $(2n+1)$-manifold $(M,\xi ,J,\theta )$, with CR structure $%
J$ and compatible contact form $\theta $, is called a CR torsion soliton if
there exist an infinitesimal contact diffeomorphism $X$ and a constant $\mu
\in {\mathbb{R}}$ such that 
\begin{equation}
\left\{ 
\begin{array}{l}
W\theta +\frac{1}{2}L_{X}\theta =\mu \theta , \\ 
L_{X}J=2A_{J,\theta },%
\end{array}%
\right.  \label{1A}
\end{equation}%
where $W$ is the Tanaka-Webster scalar curvature of $(M,\xi ,J,\theta )$ and 
$L_{X}$ denotes Lie derivative by $X$. It is called \textit{shrinking} if $%
\mu >0$, \textit{steady} if $\mu =0$, and \textit{expanding} if $\mu <0$. In
particular, it is called a CR Yamabe soliton (\cite{ccc}) if 
\begin{equation}
\left\{ 
\begin{array}{l}
W\theta +\frac{1}{2}L_{X}\theta =\mu \theta , \\ 
L_{X}J=0,%
\end{array}%
\right.  \label{2A}
\end{equation}%
where $J$ is invariant under the contact diffeomorphism.
\end{definition}

\begin{lemma}
\label{l31} A quintuple $(M^{3},J,\theta ,f,\mu )$ is a three-dimensional CR
torsion soliton if%
\begin{equation}
\left\{ 
\begin{array}{l}
W+\frac{1}{2}f_{0}=\mu , \\ 
f_{11}+iA_{11}f=-A_{11}.%
\end{array}%
\right.  \label{1a}
\end{equation}
\end{lemma}

\begin{proof}
We first recall a result from \cite{g}.

\begin{lemma}
Let $\left( M^{2n+1},J,\theta \right) $ be a pseudohermitian $(2n+1)$%
-manifold. For any smooth function $\widetilde{f}$ on $M$, let $X_{%
\widetilde{f}}$ be the vector field uniquely defined by 
\begin{equation*}
X_{\widetilde{f}}\rfloor \ d\theta =d\widetilde{f}\text{ mod }\theta\hspace{%
8mm} \mbox{ and } \hspace{8mm} X_{\widetilde{f}}\rfloor \theta =-\widetilde{f%
}.  \label{41c}
\end{equation*}
Then $X_{\widetilde{f}}=i\widetilde{f}_{\alpha }Z_{\bar{\alpha}}-i\widetilde{%
f}_{\bar{\alpha}}Z_{\alpha }-\widetilde{f}\mathbf{T}$ and it is a smooth
infinitesimal contact diffeomorphism of $\left(M^{2n+1},\theta \right) $.
Conversely, every smooth infinitesimal contact diffeomorphism is of the form 
$X_{\widetilde{f}}$ for some smooth function $\widetilde{f}$. Moreover, $%
L_{X_{\widetilde{f}}}J$ has the following expression 
\begin{equation*}
L_{X_{\widetilde{f}}}J\equiv 2(\widetilde{f}_{\alpha \alpha }+iA_{\alpha
\alpha }\widetilde{f})\theta ^{\alpha }\otimes Z_{\bar{\alpha}}+2(\widetilde{%
f}_{\bar{\alpha}\bar{\alpha}}-iA_{\bar{\alpha}\bar{\alpha}}\widetilde{f}%
)\theta ^{\bar{\alpha}}\otimes Z_{\alpha }\ \mathrm{mod}\ \theta .
\end{equation*}
\end{lemma}

Next, we relate CR torsion solitons and self-similar solutions to the CR
torsion flow. A special class of solutions to the CR torsion flow (\ref{0})
is given by self-similar solutions, whose contact forms $\theta _{t}$ deform
under the CR Yamabe flow only by a scaling function depending on $t$ 
\begin{equation}
\theta(t):=\rho (t)\Phi _{t}^{\ast }\theta(0),\ \ \ \rho (t)>0,\ \ \ \rho
(0)=1  \label{2BB}
\end{equation}%
and 
\begin{equation}
J(t)=\Phi _{t}^{\ast }(J(0)),  \label{1BB}
\end{equation}%
where $\Phi _{t}:M\rightarrow M$ is a one-parameter family of contact
diffeomorphisms generated by a CR vector field $X_{\widetilde{f}}$ as above
on $M$ with $\Phi _{0}=id_{M}$.

As in the paper \cite{ccc}, it follows from (\ref{2BB}) and $\partial_t
\theta =-2W\theta$ that we put $f=-\widetilde{f}$ and denote $\mu =-\frac{1}{%
2}\rho ^{\prime }(0),$ then we have 
\begin{equation*}
W+\frac{1}{2}f_{0}=\mu .
\end{equation*}%
On the other hand, it follows from (\ref{1BB}) that 
\begin{equation*}
\frac{\partial }{\partial t}J(t)=\frac{\partial }{\partial t}\Phi
_{t}^{\ast}(J(0))=L_{X_{\widetilde{f}}}J(t)
\end{equation*}%
and then 
\begin{equation*}
A=(\widetilde{f}_{\alpha \alpha }+iA_{\alpha \alpha }\widetilde{f})\theta
^{\alpha }\otimes Z_{\bar{\alpha}}+(\widetilde{f}_{\bar{\alpha}\bar{\alpha}%
}-iA_{\bar{\alpha}\bar{\alpha}}\widetilde{f})\theta ^{\bar{\alpha}}\otimes
Z_{\alpha }.
\end{equation*}%
Hence, when $n=1$, we obtain 
\begin{equation*}
f_{11}+iA_{11}f=-A_{11}
\end{equation*}%
for $f=-\widetilde{f}$ and $n=1$. We refer to \cite{ccc} for more details.
\end{proof}

To prove Theorem 1.2, we shall need a certain differential Harnack quantity
for the three-dimensional CR Yamabe or torsion flow. For the Ricci flow,
Hamilton found a conserved quantity which vanishes identically for expanding
Ricci solitons and showed that such a quantity is nonnegative for 
solutions to the Ricci flow with positive curvature operator. This quantity is called the
differential Harnack quantity, or LYH quantity, because one can obtain the
Harnack inequality for the evolving scalar curvature from it. Recall that
the authors have derived a CR Harnack quantity for CR Yamabe soliton in \cite%
{ccc}: 
\begin{equation}
4\Delta _{b}W+2W(W-\mu )+\langle \nabla _{b}W,X_{f}\rangle +\langle W_{0}%
\mathbf{T},X_{f}\rangle =0 .  \label{3}
\end{equation}
In this paper, we obtain the CR Harnack quantity for the CR torsion flow (%
\ref{0}).

\begin{theorem}
\label{t42} A three-dimensional CR torsion soliton satisfies the following
CR Harnack quantity%
\begin{equation}
4\Delta _{b}W+2W(W-\mu )+\langle \nabla _{b}W,X_{f}\rangle +\langle W_{0}%
\mathbf{T},X_{f}\rangle -i(A_{11,\bar{1}\bar{1}}-A_{\bar{1}\bar{1}%
,11})-2|A_{11}|^{2}=0.  \label{2}
\end{equation}%
That is, 
\begin{equation*}
3\Delta _{b}W+2W(W-\mu )+\langle \nabla _{b}W,X_{f}\rangle +\langle W_{0}%
\mathbf{T},X_{f}\rangle -2Q-2|A_{11}|^{2}=0.
\end{equation*}%
Here $Q$ is the CR $Q$-curvature (see Section $6$).
\end{theorem}

\begin{proof}
Recall that $\Delta _{b}W=W_{1\bar{1}}+W_{\bar{1}1}$. We first differentiate
the soliton equation (\ref{1B}) and obtain 
\begin{equation*}
W_{1\bar{1}}=-\frac{1}{2}f_{01\bar{1}}=\frac{i}{2}(f_{1\bar{1}}-f_{\bar{1}%
1})_{1\bar{1}}.
\end{equation*}%
The two terms appear on the right hand side are higher derivatives of $f$
which can be reduced by using 
\begin{equation*}
f_{11}+iA_{11}f=-A_{11}.
\end{equation*}%
Indeed, by using commutation relations, one derives%
\begin{equation*}
\begin{array}{ccl}
f_{1\bar{1}1} & = & f_{11\bar{1}}-if_{10}-f_{1}W \\ 
& = & -A_{11,\bar{1}}-i(A_{11}f)_{\bar{1}}-if_{01}+iA_{11}f_{\bar{1}}-f_{1}W
\\ 
& = & -A_{11,\bar{1}}-iA_{11,\bar{1}}f+2iW_{1}-f_{1}W.%
\end{array}%
\end{equation*}%
Differentiate this in the direction $Z_{\bar{1}}$, one achieves 
\begin{equation*}
f_{1\bar{1}1\bar{1}}=-A_{11,\bar{1}\bar{1}}-iA_{11,\bar{1}\bar{1}}f-iA_{11,%
\bar{1}}f_{\bar{1}}+2iW_{1\bar{1}}-f_{1\bar{1}}W-f_{1}W_{\bar{1}}.
\end{equation*}%
On the other hand, the differentiation of the conjugation in the $Z_{1}$
direction gives an expression of $f_{\bar{1}1\bar{1}1}$. After changing the $%
3$rd and $4$th indices, one obtains 
\begin{equation*}
f_{\bar{1}11\bar{1}}=f_{\bar{1}1\bar{1}1}+if_{\bar{1}10}=-A_{\bar{1}\bar{1}%
,11}+iA_{\bar{1}\bar{1},11}f+iA_{\bar{1}\bar{1},1}f_{1}-2iW_{\bar{1}1}-f_{%
\bar{1}1}W-f_{\bar{1}}W_{1}+if_{\bar{1}10}.
\end{equation*}%
Note that the bad term $f_{\bar{1}10}$ can be rewritten as%
\begin{equation*}
\begin{array}{ccl}
f_{\bar{1}10} & = & f_{0\bar{1}1}-A_{11}f_{\bar{1}\bar{1}}-A_{11,\bar{1}}f_{%
\bar{1}}-A_{\bar{1}\bar{1}}f_{11}-A_{\bar{1}\bar{1},1}f_{1} \\ 
& = & -2W_{\bar{1}1}-A_{11}f_{\bar{1}\bar{1}}-A_{11,\bar{1}}f_{\bar{1}}-A_{%
\bar{1}\bar{1}}f_{11}-A_{\bar{1}\bar{1},1}f_{1} \\ 
& = & 2|A_{11}|^{2}-2W_{\bar{1}1}-A_{11,\bar{1}}f_{\bar{1}}-A_{\bar{1}\bar{1}%
,1}f_{1}.%
\end{array}%
\end{equation*}%
Substituting our expressions for $f_{1\bar{1}1\bar{1}}$ and $f_{\bar{1}11%
\bar{1}}$ into the equation for $W_{1\bar{1}}$, we get%
\begin{equation*}
\begin{array}{ccl}
2W_{1\bar{1}} & = & -i(A_{11,\bar{1}\bar{1}}-A_{\bar{1}\bar{1},11})+A_{11,%
\bar{1}\bar{1}}f+A_{11,\bar{1}}f_{\bar{1}}-2W_{1\bar{1}}-if_{1\bar{1}%
}W-if_{1}W_{\bar{1}} \\ 
&  & +A_{\bar{1}\bar{1},11}f+A_{\bar{1}\bar{1},1}f_{1}-2W_{\bar{1}1}+if_{%
\bar{1}1}W+if_{\bar{1}}W_{1} \\ 
&  & +2|A_{11}|^{2}-2W_{\bar{1}1}-A_{11,\bar{1}}f_{\bar{1}}-A_{\bar{1}\bar{1}%
,1}f_{1}.%
\end{array}%
\end{equation*}%
By the CR Bianchi identity $A_{11,\bar{1}\bar{1}}+A_{\bar{1}\bar{1}%
,11}=W_{0} $, we have the Harnack quantity%
\begin{equation}
\begin{array}{ccl}
4\Delta _{b}W & = & -i(A_{11,\bar{1}\bar{1}}-A_{\bar{1}\bar{1}%
,11})+2|A_{11}|^{2}+(A_{11,\bar{1}\bar{1}}+A_{\bar{1}\bar{1},11})f \\ 
&  & -i(f_{1\bar{1}}-f_{\bar{1}1})W-i(f_{1}W_{\bar{1}}-f_{\bar{1}}W_{1}) \\ 
& = & -i(A_{11,\bar{1}\bar{1}}-A_{\bar{1}\bar{1},11})+2|A_{11}|^{2}+W_{0}f
\\ 
&  & -i(if_{0})W-i(f_{1}W_{\bar{1}}-f_{\bar{1}}W_{1}) \\ 
& = & -i(A_{11,\bar{1}\bar{1}}-A_{\bar{1}\bar{1},11})+2|A_{11}|^{2}+W_{0}f
\\ 
&  & -2W(W-\mu)+\langle \nabla _{b}W,J(\nabla _{b}f)\rangle _{\theta }.%
\end{array}%
\end{equation}

Notice that 
\begin{align*}
\langle \nabla _{b}W+W_{0}\mathbf{T},X_{f}\rangle & =\langle \nabla
_{b}W+W_{0}\mathbf{T},-f\mathbf{T}-if_{\bar{1}}Z_{1}+if_{1}Z_{\bar{1}}\rangle
\\
& =-W_{0}f-\langle \nabla _{b}W,J(\nabla _{b}f)\rangle _{\theta },
\end{align*}
so the proof is completed.
\end{proof}

%%%%%%%%%%%%%%%%%%%%%%%%%%%%%%%%%%%%%%%%%%%%%%%%%%%%%

\section{Classification of CR Solitons}

%%%%%%%%%%%%%%%%%%%%%%%%%%%%%%%%%%%%%%%%%%%%%%%%

In the first part of this section, we derive a classification of closed
three-dimensional CR Yamabe solitons according to pseudohermitian curvature
condition. In the second part, we prove the complete classification of
closed three-dimensional CR torsion solitons.

In \cite{ccc}, we use Harnack quantity (\ref{3}) to prove that every closed
three-dimensional CR Yamabe soliton satisfies $\int_M(W-\mu)^2 =0$ and thus
has constant Tanaka-Webster curvature. This is also true for higher dimensional
cases, as proved by P.-T. Ho in \cite{ho}. However, this is not enough
for us to classify the underlying CR manifold due to the lack of information
on CR torsion. We suggest that $C_0$-positivity, which involves both
curvature and torsion, may be a suitable notion for achieving a complete
classification of closed three-dimensional CR Yamabe solitons.

\begin{theoremt41}
\textrm{(i)} Any simply connected closed three-dimensional CR Yamabe soliton
with $(W+Tor)(X,X)>0$ for $0\neq X\in T_{1,0}(M)$ and $P_{0}f=0$ must be the
standard CR three sphere.

\textrm{(ii)} Any closed three-dimensional CR Yamabe soliton with $%
(W+Tor)(X,X)=0$ for $0\neq X\in T_{1,0}(M)$ and $P_{0}f=0$ must be the
standard Heisenberg space form which is a Seifert fiber space over a
euclidean $2$-orbifold with nonzero Euler number.

\textrm{(iii)} Any closed three-dimensional CR Yamabe soliton with $%
(W+Tor)(X,X)<0$ for $0\neq X\in T_{1,0}(M)$ and nonnegative CR Paneitz
operator must be the standard Lorentz space form which is a Seifert fiber
space over a hyperbolic orbifold with nonzero Euler number.
\end{theoremt41}

%\textbf{Proof of Theorem \ref{t41}\ :}

\begin{proof}
(i) Since $f_{11}+iA_{11}f=0$, by using commutation relations, one derives 
\begin{equation*}
\begin{array}{ccl}
f_{1\bar{1}1} & = & f_{11\bar{1}}-if_{10}-f_{1}W \\ 
& = & -i(A_{11}f)_{\bar{1}}-if_{01}+iA_{11}f_{\bar{1}}-f_{1}W \\ 
& = & -iA_{11,\bar{1}}f+2iW_{1}-f_{1}W.%
\end{array}%
\end{equation*}%
Recall that $W=\mu$, so the term $W_1$ vanishes (see Theorem 1.1 in \cite%
{ccc}). Multiplying both sides by $f_{\bar{1}}$ and integrating them, we
have 
\begin{equation*}
-\int_{M}|f_{1\bar{1}}|^{2}d\mu +i\int_{M}A_{11,\bar{1}}f_{\bar{1}%
}fd\mu+\int_{M}Wf_{\bar{1}}f_{1}d\mu=0.
\end{equation*}
Because 
\begin{align*}
i\int_{M}A_{11,\bar{1}}f_{\bar{1}}fd\mu =& -i\int_{M}A_{11}f_{\bar{1}}f_{%
\bar{1}}d\mu -i\int_{M}A_{11}f_{\bar{1}}{}_{\bar{1}}fd\mu \\
=& -i\int_{M}A_{11}f_{\bar{1}}f_{\bar{1}}d\mu +\int_{M}|A_{11}|^{2}f^{2}d\mu,
\label{C}
\end{align*}
we have 
\begin{equation*}
0=\int_{M}|A_{11}|^{2}f^{2}d\mu -\int_{M}|f_{1\bar{1}}|^{2}d\mu
-i\int_{M}A_{11}f_{\bar{1}}f_{\bar{1}}d\mu +\int_{M}Wf_{\bar{1}}f_{1}d\mu.
\end{equation*}%
Using the soliton equation and the pluriharmonic operator $P_{1}\varphi
=\varphi _{\bar{1}}{}^{\bar{1}}{}_{1}+iA_{11}\varphi ^{1}$, we can break the
second term into 
\begin{align*}  \label{D}
\int_{M}|f_{1\bar{1}}|^{2}d\mu = \int_{M}|f_{\bar{1}1}|^{2}d\mu +i\int_{M}f_{%
\bar{1}1}{}f_{0}d\mu = i\int_{M}A_{11}f_{\bar{1}}f_{\bar{1}}d\mu
-\int_{M}(P_{1}f)f_{\bar{1}}d\mu.
\end{align*}
All these imply%
\begin{equation*}
0=\int_{M}|A_{11}|^{2}f^{2}d\mu+\int_{M}Wf_{\bar{1}}f_{1}d\mu
-2i\int_{M}A_{11}f_{\bar{1}}f_{\bar{1}}d\mu +\int_{M}(P_{1}f)f_{\bar{1}}d\mu.
\end{equation*}%
By taking conjugate and summing with the original equation, we have 
\begin{equation*}
0 = \int_{M}|A_{11}|^{2}f^{2}d\mu+\int_{M}Wf_{\bar{1}}f_{1}d\mu-i%
\int_{M}(A_{11}f_{\bar{1}}f_{\bar{1}}-A_{\bar{1}\bar{1}}f_{1}f_{1})d\mu -%
\frac{1}{2}\int_{M}(P_{0}f)fd\mu,
\end{equation*}
i.e., 
\begin{equation}  \label{intHarYamabe}
0= \int_{M}|A_{11}|^{2}f^{2}d\mu +\int_{M}(W+Tor)((\nabla _{b}f)_{\mathbb{C}%
},(\nabla _{b}f)_{\mathbb{C}})d\mu -\frac{1}{2}\int_{M}(P_{0}f)fd\mu.
\end{equation}

By our assumptions and the identity (\ref{intHarYamabe}), it is easy to see $%
A_{11}=0$. Since a pseudohermitian $3$-manifold $(M,J,\theta )$ of constant
Tanaka-Webster scalar curvature and vanishing pseudohermitian torsion must
be spherical, it follows from a result of Y. Kamishima and T. Tsuboi (\cite{kt})
that one can have a complete classification of such closed spherical
torsion-free CR torsion solitons. We also refer to \cite{t} for closed CR $3$%
-manifolds of Sasakian space forms.

The cases (ii) and (iii) can be similarly derived as (i).
\end{proof}

%\begin{remark}
The condition $P_{0}f=0$ in Theorem \ref{t41} is not very restrictive in the
sense that the kernel of the CR Paneitz operator $P_{0}$ is infinite
dimensional, containing all CR-pluriharmonic functions (Cf.\cite{hi}). For a
closed pseudohermitian $3$-manifold of transverse symmetry, we have $\ker
P_{1}=\ker P_{0}.$ Moreover, these kernels are invariant under rescaling $%
\widetilde{\theta }=e^{2g}\theta $, since $\widetilde{P}_{1}=e^{-3g}P_{1}$
and $\widetilde{P}_{0}=e^{-4g}P_{0}$. Note that CR-pluriharmonic function is
naturally related to holomorphic functions on $\mathbb{C}^{2}$. In fact, let 
$M$ be a hypersurface in $\mathbb{C}^{2}$, i.e., $M=\partial \Omega $ for a
bounded domain $\Omega $ in $\mathbb{C}^{2}$, then for any pluriharmonic
function $u:\mathbf{U}\rightarrow \mathbb{R}$ ($\partial \bar{\partial}u=0$)
with a simply connected $\mathbf{U}\subset \bar{\Omega}$, there exists a
holomorphic function $w$ in $\mathbf{U}$ such that $u=\mathrm{Re}(w)$. Now
define $f:=u|_{M}$, it follows that $f$ is a CR-pluriharmonic function (see
Definition \ref{4}) and thus $P_{0}f=0$.\newline
%\end{remark}

We have mentioned that every closed three-dimensional CR Yamabe soliton has
constant Tanaka-Webster curvature. This can be proven by integrating the CR
Yamabe Harnack quantity (\ref{3}). For CR torsion solitons, we obtain
similar results by integrating the CR torsion Harnack quantity (\ref{2}).
Therefore we obtain the following classification.

\begin{theoremtA}
Every closed three-dimensional CR torsion soliton (\ref{1B}) has constant
Tanaka-Webster curvature and zero torsion $A_{11}=0$. Therefore, they must
belong to one of the following three cases:

\textrm{(i)} The standard CR spherical space form in case $W=1$.

\textrm{(ii)} The standard Heisenberg space form in case $W=0$.

\textrm{(iii)} The standard Lorentz space form in case $W=-1$.
\end{theoremtA}

%\textbf{Proof of Theorem \ref{tA}\ : }

\begin{proof}
By integrating the Harnack quantity (\ref{2}), one derives that 
\begin{equation*}
\begin{array}{ccl}
2\int_{M}|A_{11}|^{2}d\mu & = & \int_{M}[2W(W-\mu )-W_{0}f+i(f_{1}W_{\bar{1}%
}-f_{\bar{1}}W_{1})]d\mu \\ 
& = & 2\int_{M}W(W-\mu )d\mu +\int_{M}Wf_{0}d\mu -\int_{M}i(f_{1\bar{1}}-f_{%
\bar{1}1})Wd\mu \\ 
& = & 2\int_{M}W(W-\mu )d\mu +2\int_{M}Wf_{0}d\mu \\ 
& = & 2\int_{M}W(W-\mu )d\mu -4\int_{M}W(W-\mu )d\mu \\ 
& = & -2\int_{M}W(W-\mu )d\mu .%
\end{array}%
\end{equation*}%
Together with the fact that $\int_{M}(W-\mu )d\mu =-\frac{1}{2}%
\int_{M}f_{0}d\mu =0$, we obtain%
\begin{equation*}
\begin{array}{cll}
\int_{M}(W-\mu )^{2}d\mu & = & \int_{M}W(W-\mu )d\mu -\mu \int_{M}(W-\mu
)d\mu \\ 
& = & \int_{M}W(W-\mu )d\mu \\ 
& = & -\int_{M}|A_{11}|^{2}d\mu .%
\end{array}%
\end{equation*}%
Hence $W=\mu$ and $A_{11}=0$. This completes the proof.
\end{proof}

%%%%%%%%%%%%%%%%%%%%%%%%%%%%%%%%%%%%%%%%%%%%%%%%

\section{Curvature Evolution Equations and Preserving Positive Curvature }

%%%%%%%%%%%%%%%%%%%%%%%%%%%%%%%%%%%%%%%%%%%%%%%%

Let $(M^{3},J,\theta )$ be a closed CR manifold. Recall that, as in
Definition 2.1, a function $u$ is called pluriharmonic w.r.t. $\theta $ if 
\begin{equation*}
Pu:=(u_{\bar{1}}{}^{\bar{1}}{}_{1}+iA_{11}u^{1})\theta ^{1}=(P_{1}u)\theta
^{1}=0.
\end{equation*}%
We denote $W^{\perp }=W-W^{\ker }$, where $W^{\ker }$ is the pluriharmonic
portion w.r.t. $\theta $, i.e., $PW^{\ker }=0$ and$\ PW=PW^{\perp }$. Note
that $\theta $ and $\hat{\theta}:=e^{-2\gamma }\theta $ characterize the
same pluriharmonic functions, because $P=e^{-4\gamma }\hat{P}$.

In Theorem 1.1, we have seen that $C_{0}$-positivity is equivalent to the
pointwise pinching condition $W>2C_{0}|A_{11}|$. Here we show that the
uniform pinching condition $W>2C_{0}\max |A_{11}|$ is preserved by the CR
torsion flow with specific initial data. To be precise, given a background
CR manifold $(M^{3},J,\overset{0}{\theta })$ which is pseudo-Einstein, if
there exists a pluriharmonic function $\gamma (0)$ such that $W(0)$ of $%
\theta (0):=e^{2\gamma (0)}\overset{0}{\theta }$ is pluriharmonic, then the
CR torsion flow 
\begin{equation}
\left\{ 
\begin{array}{l}
\frac{\partial J}{\partial t}=2A_{J,\theta }, \\ 
\frac{\partial \theta }{\partial t}=-2W\theta , \\ 
\overset{0}{P}\gamma (t)=0,\ \overset{0}{P}W(t)=0%
\end{array}%
\right.   \label{torsionflow}
\end{equation}%
has a short-time solution $\theta (t)=e^{2\gamma (t)}\overset{0}{\theta }$
on $[0,T)$ which preserves the uniform pinching condition.

Note that the initial contact form $\theta (0)$ of the flow differs from the
background contact form $\overset{0}{\theta }$ by the factor $e^{2\gamma (0)}
$. According to Lee \cite{l1} and Hirachi \cite{hi}, $\theta (0):=e^{2\gamma
(0)}\overset{0}{\theta }$ is still pseudo-Einstein because $\gamma (0)$ is a
pluriharmonic function. This \textquotedblleft gauge changing" plays an
analogue role as DeTurck's trick of the Ricci flow in Chang and Wu's proof
of the existence of the torsion flow $\theta (t)$. Moreover, $\overset{0}{P}%
\gamma (t)=0$ and $\overset{0}{P}W(t)=0$ imply that $\gamma (t)=\frac{1}{2}%
\left( \ln \theta (t)-\ln \overset{0}{\theta }\right) $ and $W(t)$ are still
pluriharmonic for all $t\in \lbrack 0,T)$\ (\cite[p.20]{cw}). This makes the
evolution equations of curvature and torsion easier to handle.
Our result in this section is based on these facts. For more properties
about this CR torsion flow, one can consult \cite{ckw,cw}.

Before we prove the preservation of the uniform pinching condition, we
clarify the definition of pseudo-Einstein mentioned above. Traditionally,
the contact form $\theta $ of a CR manifold $(M^{2n+1},J,\theta )$ is called
pseudo-Einstein if $R_{\alpha \bar{\beta}}-\frac{1}{n}Rh_{\alpha \bar{\beta}%
}=0$. Note that this condition holds trivially when $n=1$ and thus gives no
information about $\theta $. When $n\geq 2$, pseudo-Einstein is equivalent
to $W_{\alpha }-iA_{\alpha \beta ,\bar{\beta}}=0$. However, $W_{\alpha }-iA_{\alpha
\beta ,\bar{\beta}}$ does not always vanishes when $n=1$. This can be seen as an
alternative definition of pseudo-Einstein. So in this article
a CR manifold $(M^{3},J,\theta )$ is called \textit{pseudo-Einstein} if 
\begin{equation*}
H_{1}:=W_{1}-iA_{11,\bar{1}}=0.
\end{equation*}%
Note that to be pseudo-Einstein is neither sufficient to be a self-similar
solution of CR torsion flow or CR Yamabe flow, nor sufficient to conclude
that $W(0)$ is a constant.

\begin{remark}
1. There always exists a pseudo-Einstein contact form on a closed
hypersurface in $\mathbb{C}^{2}$.

2. Let $(M,J,\theta )$ be a closed CR $3$-manifold. It is still open whether
there exists a pseudo-Einstein contact form $\hat{\theta}$ of pluriharmonic
scalar curvature (in particular, the constant scalar curvature)! We refer to 
\cite{cklin} for some details.
\end{remark}

Now we demonstrate the preservation of the uniform pinching condition.

\begin{theoremt61}
Let $(M^{3},J,\overset{0}{\theta })$ be a closed CR manifold with pseudo-Einstein $\overset{0}{\theta}$. 
If there exists a pluriharmonic function $\gamma(0)$ such that $W(0)$ of 
$\theta (0):=e^{2\gamma (0)}\overset{0}{\theta }$ is pluriharmonic, then the solution $\theta(t)$ of the
CR torsion flow (\ref{torsionflow}) has the following evolution equations 
\begin{equation*}
\left\{ 
\begin{array}{ccl}
\frac{\partial W}{\partial t} & = & 5\Delta _{b}W+2(W^{2}-|A_{11}|^{2}), \\ 
\frac{\partial |A_{11}|^{2}}{\partial t} & = & \Delta
_{b}|A_{11}|^{2}-2|\nabla |A_{11}||^{2}.%
\end{array}%
\right. 
\end{equation*}%
Moreover, if $W>2C_{0}\max |A_{11}|$ initially with $C_{0}\geq \frac{1}{2}$,
then it holds for all $t<T$. In particular, the solution $\theta (t)$ is $%
C_{0}$-positive for all $t<T$, i.e., 
\begin{equation*}
(W+C_{0}Tor)(X,X)>0\ \mbox{ for all non-vanishing}\ X\in T_{1,0}(M)\mbox{
and for all }t<T.
\end{equation*}
\end{theoremt61}

\begin{proof}
For generic CR torsion flow, we have (\cite[pp.12-13]{cw}) 
\begin{equation*}
\frac{\partial }{\partial t}{A}_{\bar{1}\bar{1}}=2(iW_{\bar{1}\bar{1}}+W{A}_{%
\bar{1}\bar{1}})-iA{_{\bar{1}\bar{1}}},_{0}
\end{equation*}%
and 
$$\frac{\partial }{\partial t}{W}=  2\mathrm{Re}\left( iA_{11,\bar{1}\bar{%
1}}-A_{11}A_{\bar{1}\bar{1}}\right) +(4\Delta _{b}W+{2W}^{2}).%
$$
Hence we have 
\begin{equation*}
\left\{ 
\begin{array}{ccl}
\frac{\partial W}{\partial t} & = & 4\Delta
_{b}W+2(W^{2}-|A_{11}|^{2})+i(A_{11,\bar{1}\bar{1}}-A_{\bar{1}\bar{1},11}),
\\ 
\frac{\partial A_{11}}{\partial t} & = & -2iW_{11}+A_{11,1\bar{1}}-A_{11,%
\bar{1}1},%
\end{array}%
\right. 
\end{equation*}

Recall that CR $Q$-curvature is defined by 
\begin{equation*}
Q:=-\mathrm{Re}(W_{1}-iA_{11,\bar{1}})_{\bar{1}}=-\mathrm{Re}(W_{1\bar{1}%
}-iA_{11,\bar{1}\bar{1}})=-\frac{1}{2}[\Delta _{b}W-i(A_{11,\bar{1}\bar{1}%
}-A_{\bar{1}\bar{1},11})].
\end{equation*}%
If we denote $H_{1}:=W_{1}-iA_{11},_{\bar{1}}$, then $H_{1\bar{1}}+H_{\bar{1}%
1}=-2Q$ and 
\begin{equation}
\left\{ 
\begin{array}{ccl}
\frac{\partial W}{\partial t} & = & 5\Delta _{b}W+2(W^{2}-|A_{11}|^{2})+2Q,
\\ 
\frac{\partial |A_{11}|^{2}}{\partial t} & = & \Delta
_{b}|A_{11}|^{2}-2|\nabla |A_{11}||^{2}+4\mathrm{Im}(H_{11}A_{\bar{1}\bar{1}%
}).%
\end{array}%
\right.   \label{01}
\end{equation}

Since $\theta(t)$ is pseudo-Einstein for all $t$, we have $H_{1}=0$, thus $%
Q=0$ and $4\mathrm{Im}(H_{11}A_{\bar{1}\bar{1}})=0$. Therefore, we obtain the desired evolution equation. Now applying the maximum principle to the second
evolution equation, one can show that 
\begin{equation*}
|A_{11}(t)|\leq \max |A_{11}(0)|
\end{equation*}%
for all $t<T$. But from the first evolution equation, we have%
\begin{equation*}
\begin{array}{ccl}
\frac{\partial W}{\partial t} & \geq & 5\Delta _{b}W+2(W^{2}-\max
|A_{11}(0)|^{2}) \\ 
& = & 5\Delta _{b}W+2(W+\max |A_{11}(0)|)(W-\max |A_{11}(0)|).%
\end{array}%
\end{equation*}%
It follows from the maximal principle that the condition $W(t)>\max
|A_{11}(t)|$ is preserved under the flow and 
\begin{equation*}
W(t)>\max |A_{11}(0)|\geq \max |A_{11}(t)|,
\end{equation*}%
which is equivalent to say that 
\begin{equation*}
(W+\frac{1}{2}Tor)(Z,Z)>0,\ \ \ Z=x^{1}Z_{1}\in T_{1,0}(M).
\end{equation*}%
holds for all $t<T.$
\end{proof}

\begin{remark}
1. Note that if $M$ is an embeddable CR $3$-manifold in $\mathbb{C}^{N}$,
then there exists a contact form $\overset{0}{\theta }$ of vanishing CR $Q$%
-curvature $\overset{0}{Q}=0$ (\cite{cs}). Hence $Q=0$ and the generic
evolution equations (\ref{01}) become 
\begin{equation*}
\left\{ 
\begin{array}{ccl}
\frac{\partial W}{\partial t} & = & 5\Delta _{b}W+2(W^{2}-|A_{11}|^{2}), \\ 
\frac{\partial |A_{11}|^{2}}{\partial t} & = & \Delta
_{b}|A_{11}|^{2}-2|\nabla |A_{11}||^{2}+4\mathrm{Im}(H_{11}{A}_{\bar{1}\bar{1%
}}).%
\end{array}%
\right. 
\end{equation*}

2. Under the same condition $W>2C_0\max |A_{11}|$, we also obtain that $%
(W-C_0Tor)(X,X)>0$ for all non-vanishing $X\in T_{1,0}(M)$ and for all $t<T$.
\end{remark}

The torsion flow greatly simplifies if the torsion vanishes. This only
happens in very special setups. Indeed, CR $3$-manifolds with vanishing
torsion are $K$-contact, meaning that the Reeb vector field is a Killing
vector field for the contact Riemannian metric $g=\frac{1}{2}d\theta
+\theta^{2}$. In general, one can still hope that the torsion flow improves
properties of the contact manifold underlying the CR structure. It is the
case in a closed homogeneous pseudohermitian $3$-manifold (\cite{ckw}).
Long-time existence and asymptotic convergence of solutions for the
(normalized) torsion flow can be achieved in this special case.

\end{document}